\documentclass{amsart}

\usepackage{lineno, hyperref}
\usepackage{amsthm}
\usepackage{amsmath, amssymb}
\usepackage[msc-links, nobysame, non-sorted-cites]{amsrefs}
\usepackage[all]{xy}
\usepackage{calrsfs}
\usepackage{mathpazo}

\newcommand{\R}{\mathbb R}

\newcommand{\Z}{\mathbb Z}

\newcommand{\X}{\mathfrak X}

\DeclareMathOperator{\diff}{Diff} 
\DeclareMathOperator{\pdiff}{PDiff}
\DeclareMathOperator{\mcg}{Mod}
\DeclareMathOperator{\pmcg}{PMod}

\DeclareMathOperator{\Ima}{Im}
\DeclareMathOperator{\lcm}{lcm}

\DeclareMathOperator{\cd}{cd}
\DeclareMathOperator{\vcd}{vcd}
\DeclareMathOperator{\per}{p} 

\DeclareMathOperator{\phome}{PHomeo}
\DeclareMathOperator{\Aut}{Aut}

\newcommand{\modn}{\mathcal{N}_g^k}     
\newcommand{\modno}{\mathcal{N}_g^1}    
\newcommand{\hiper}{\mathbb{H}^2}       
\newcommand{\Zp}{\Z/p}                  
\newcommand{\diffnk}{\pdiff(N_g;k)}     
\newcommand{\homek}{\phome(N_g;k)}      

\newcommand{\n}{\noindent}

\newtheorem{thm}{Theorem}[section]
\newtheorem{lem}[thm]{Lemma}
\newtheorem{prop}[thm]{Proposition}

\newtheorem{mainthm}{Theorem}
\newtheorem{maincor}[mainthm]{Corollary}

\theoremstyle{definition}

\newtheorem{rmk}[thm]{Remark}

\newcommand{\cuadro}[8]{\xymatrix @C=1.3cm @M=2mm{
	{#1} \ar[r]^-{#2} \ar[d]_-{#4} & {#3} \ar[d]^-{#5}  \\
	{#6} \ar[r]^-{#7} & {#8}   }}

\begin{document}

\title{Periodicity of the pure mapping class group \\ of non-orientable surfaces}

\author{Nestor Colin}
\address{Instituto de Matem\'aticas, Universidad Nacional Aut\'onoma de M\'exico
Oaxaca de Ju\'arez, Oaxaca 68000, M\'exico}
\email{ncolin@im.umam.mx} 
\email{rita@im.umam.mx}

\author{Rita Jiménez Rolland}

\author{Miguel A. Xicot\'encatl} 
\address{Departamento de Matem\'aticas, Centro de Investigaci\'on y de Estudios Avanzados del IPN, 
Mexico City 07360, Mexico}
\email{xico@math.cinvestav.mx}




\maketitle

\markboth{\sc{Nestor Colin, Rita Jiménez Rolland and Miguel A. Xicot\'encatl}}{\sc{Periodicity of the pure mapping class group of non-orientable surfaces}}

\begin{abstract}

We show that the pure mapping class group $\mathcal{N}_{g}^{k}$ of a non-orientable closed surface of genus $g\geqslant 2$ with $k\geqslant 1$ marked points has $p$-periodic cohomology for each odd prime $p$ for which $\mathcal{N}_{g}^{k}$ has $p$-torsion. Using the Yagita invariant and the cohomology classes obtained by the representation of subgroups of order $p$, we obtain that the $p$-period is less than or equal to $4$ when $g\geqslant 3$ and $k\geqslant 1$. Moreover, combining the Nielsen realization theorem and a characterization of the $p$-period given in terms of normalizers and centralizers of cyclic subgroups of order $p$, 
we show that the $p$-period of $\mathcal{N}_{g}^{k}$ is bounded below by $4$, whenever $\modn$ has $p$-periodic cohomology, $g\geqslant 3$ and $k\geqslant 0$. These results provide partial answers to questions 
proposed by G. Hope and U. Tillmann.
\end{abstract}


\section{Introduction}

Let $\Gamma$ be a group of finite virtual cohomological dimension ($\vcd$) and let $p$ be a prime. The group $\Gamma$ is called {\it $p$-periodic} if there exists a positive integer $d$ such that the Farrell cohomology groups $\widehat{H}^i(\Gamma;M)$ and $\widehat{H}^{i+d}(\Gamma;M)$  have naturally isomorphic $p$-primary components for all $i\geqslant 0$ and for all $\mathbb{Z}\Gamma$-modules $M$. 
The least of such $d$ is called the {\it $p$-period} of $\Gamma$ and is denoted by $\per(\Gamma)$. Farrell cohomology extends Tate cohomology of finite groups to groups of finite $\vcd$ and in degrees above the $\vcd$ it agrees with the ordinary cohomology of the group.  In this paper we study the $p$-periodicity of the pure mapping class group of a non-orientable closed surface with at least one marked point. 

Let $\Sigma$ be a closed surface and $\{z_1, \dots, z_k\}$ a set of $k\geqslant 0$ distinct points in $\Sigma$, we call them {\it marked points}. Let $\diff(\Sigma ; k)$ be the group of diffeomorphisms of $\Sigma$ that preserve the set of marked points and let $\pdiff(\Sigma ; k)$ be the subgroup of diffeomorphisms that fix the marked points pointwise.  
If the surface $\Sigma$ is orientable, we consider the corresponding subgroups $\diff^{\,+}(\Sigma; k )$ and $\pdiff^{\,+}(\Sigma; k )$ of orientation-preserving diffeomorphisms.  The {\it pure mapping class group}  of $\Sigma$ with $k$ marked points is the group of isotopy classes of $\pdiff(\Sigma; k)$ if $\Sigma$ is non-orientable and the group of isotopy classes of $\pdiff^{\,+}(\Sigma; k )$ when $\Sigma$ is orientable.
We use the notation $\Gamma_g^k:=\pmcg(S_g;k)$ and $\mathcal{N}_g^k:=\pmcg(N_g;k)$, where $S_g$ and $N_g$ denote, respectively, a closed connected orientable and non-orientable surface of genus $g$. If the set of marked points
is empty, we omit $k$ from the notation.

It is well known that the groups $\Gamma_g^k$ and $\mathcal{N}_g^k$ have finite $\vcd$ and their Farrell cohomology and $p$-periodicity have been previously studied in the literature. For instance, it is known that for an orientable closed surface of genus $g>1$ the group $\Gamma_g$ is never $2$-periodic, and for $p$ an odd prime Y. Xia determined in \cite{Xia92} all genera $g$ for which $\Gamma_g$ is $p$-periodic. In \cite{GMX92}  H.H. Glover, G. Mislin and Y. Xia obtained a formula for the $p$-period $\per(\Gamma_g)$ that holds whenever the group is $p$-periodic. In contrast, Q. Lu proved in \cite{lu01per} that for an orientable surface of genus $g\geqslant 1$ with at least one marked point, the group $\Gamma_g^k$ is always $p$-periodic with $p$-period equal to $2$. Using these results, Xia \cites{Xia92Farr, Xia92Farr2, Xia95} and Lu \cite{lu01per} determined the $p$-primary component of the Farrell cohomology $\widehat{H}^*(\Gamma_{p-1}^k; \Z)_{(p)}$, $\widehat{H}^*(\Gamma_{(p-1)/2}^k; \Z)_{(p)}$ and $\widehat{H}^*(\Gamma_{p}; \Z)_{(p)}$ where $k\geqslant 0$ and $p$ is an odd prime. Furthermore, Lu \cite{LuFarrell} obtained all the $p$-primary components of the Farrell cohomology of the pure mapping class group $\Gamma_g^k$ of a surface of low genus $g=1,2,3,$ when $\Gamma_g^k$ has $p$-torsion, $p$ is an odd prime, and $k\geqslant 1$. 

On the other hand,  G. Hope and U. Tillmann  investigated in \cite{HT09} the $p$- periodicity of the Farrell cohomology of the mapping class group $\mathcal{N}_g$ of a closed non-orientable surface of genus $g\geqslant 3$.  They were able to determine the precise conditions required for this cohomology to exhibit $p$-periodic behavior.  We contribute to the understanding of the $p$-periodicity of the pure mapping class group $\modn$ of a closed non-orientable surface with at least one marked point by proving the following results (see \cite{HT09}*{Question 5.2}).

\begin{mainthm}\label{Thm:Main 1 pperiodicity}
    Let $g\geqslant 2$, $k\geqslant 1$ and $p$ be an odd prime. The pure mapping class group $\modn$ has $p$-periodic cohomology whenever $\modn$ contains $p$-torsion.
\end{mainthm}

As we observe in Remark \ref{Rmk:2Periodicity} below, the argument from \cite[Proof of Lemma 4.1]{HT09} actually shows that $\modno$ is not $2$-periodic for  $g\geqslant 3$.
Using the Yagita invariant and adapting Lu's methods from \cite{lu01per} to the non-orientable case $\modn$, we find an upper bound for the $p$-period in Theorem \ref{Thm:UpperBound}, proving that $\per(\modn) \leqslant 4$ whenever the group has $p$-torsion, $g\geqslant 3$ and $k\geqslant 1$. In addition, using a different description of the $p$-period given in terms of the normalizers and centralizers of subgroups of order $p$ of $\modn$, 
we prove in Theorem \ref{Thm:LowerBound} that $\per(\modn) \geqslant 4$ whenever $\modn$ has $p$-periodic cohomology and $k\geqslant 0$.
Thus, the following result follows.

\begin{mainthm}\label{Thm:Main 2 pperiod}
    Let $g\geqslant 3$, $k\geqslant 1$ and $p$ be an odd prime. If the pure mapping class group $\modn$ contains $p$-torsion, then its $p$-period  $\per(\modn)$ is $4$.
\end{mainthm}

It is worth pointing out that the lower bound $\per(\modn)\geqslant 4$ that we find in Theorem \ref{Thm:LowerBound} applies also to the case without marked points $k=0$. Combining this result with \cite[Theorem 1.1]{HT09} yields the following result, which partially solves  \cite[Question 5.1]{HT09} about finding a lower bound for the $p$-period of $\mathcal{N}_g$.

\begin{maincor}\label{Cor:Main LowerBound}
    Let $p$ be an odd prime and suppose that the mapping class group $\mathcal{N}_g$ contains $p$-torsion. Then $\per(\mathcal{N}_g) \geqslant 4$ unless that
	$g=lp+2$ for some $l>0$, and for $0\leqslant t < p$ with $l\equiv -t \mod{p}$ we have that $l+t+2p>tp$.
\end{maincor}

\bigskip

\n
{\bf Outline.} The paper is organized as follows. In Section \ref{Sec:pperiodicity} we use the connection between the mapping class groups of a non-orientable surface and its orientable double cover, and the result of Q. Lu \cite[Theorem 1.7]{lu01per} to prove Theorem \ref{Thm:Main 1 pperiodicity}. In Section \ref{Sec:UpperBound} we recall the definition of the Yagita invariant and we use cohomology classes of subgroups of order $p$ to prove that this invariant is lower than or equal to $4$.  This gives us the upper bound for the $p$-period of $\modno$ in Theorem \ref{Thm:UpperBound OnePoint} and a Birman sequence argument allows us to obtain the upper bound for general $k\geqslant 1$ in Theorem \ref{Thm:UpperBound}. Finally, in Section \ref{Sec:Lower Bound} we prove that $4$ is also a lower bound of the $p$-period of $\modn$, when $k\geqslant 0$; see Theorem \ref{Thm:LowerBound}. For this we use Theorem \ref{Thm:GMX pperiod}, a characterization of the $p$-period given in terms of the index of the normalizers and centralizers of $\modn$ of subgroups of order $p$ over their conjugacy classes. 

One of the main ingredients to obtain the lower bound is Theorem \ref{Thm:f y f-1 conjugados}, which states that $f$ and $f^{-1}$ are conjugated in $\homek$ whenever $f\in \diffnk$ is an element of order $p$. Its proof involves a deep analysis of the automorphisms of non-Euclidean crystallographic groups (see, for instance, \cite[Section 3]{Bujalance15}). It is based on the classic work of J. Nielsen in \cite{Nielsen37}, in which he introduced the notion of fixed point data of a finite order homeomorphism of an orientable surface and characterized its conjugacy classes using this notion.

\bigskip


\section{The p-periodicity of the pure mapping class group}
\label{Sec:pperiodicity}

Let $p$ be an odd prime.  The main purpose of this section is to show that the pure mapping class group  $\modn$ has $p$-periodic cohomology, whenever $\modn$
contains $p$-torsion, for $g \geqslant 2$ and $k \geqslant 1$. The proof relies on a result of Q. Lu \cite[Theorem 1.7]{lu01per}, which states that, in the orientable case, the group
$\Gamma_g^k$ has $p$-periodic cohomology whenever the group contains $p$-torsion, for $g, k \geqslant 1$. It also uses the relation between the mapping class group of a non-orientable surface and the mapping class group of its orientable double cover that we recall next.

For $N_g$ a closed connected non-orientable surface of genus $g$,  the non-orientable double cover can
be constructed (up to isomorphism) as follows. Let $S_{g-1}$ be a closed orientable surface of genus $g-1$, embedded in $\R^3$ such that $S_{g-1}$ 
is invariant under reflections in the $xy-$, $yz-$, and $xz-$ planes. Let $\sigma :S_{g-1} \to S_{g-1}$ be the orientation reversing 
homeomorphism  
$$  \sigma(x, y, z)  =  (-x, -y, -z) .$$

Then the quotient  $S_{g-1}  /\langle \sigma\rangle$ is homeomorphic to $N_g$ and the natural projection  $\pi : S_{g-1}  \to N_g$ is 
a double cover of $N_g$ such that $\sigma$ becomes a covering transformation
$$
\xymatrix{
 S_{g-1}  \ar[rr]^\sigma \ar[rd]_\pi  & &S_{g-1}   \ar[ld]^\pi \\ 
 &  N_g  &  
 }
$$
The following result is well-known (see for example \cite[Lemma 2.2]{GGM18}).

\begin{lem}
Every diffeomorphism $f : N_g \to N_g$ admits exactly two liftings  $S_{g-1} \to S_{g-1}$, one of which
preserves orientation.
\end{lem}

\medskip

Furthermore, in the case $f \in \diff(N_g ; k)$, if $\widetilde f: S_{g-1} \to S_{g-1}$ is the orientation preserving lifting of $\pi$, 
then $\widetilde f \in \diff^+(S_{g-1} ; 2k)$. Namely, if $\{ z_1, \dots, z_k \}  \subset N_g$ is the set of marked points, 
let $\tilde z_1, \dots, \tilde z_k \in S_{g-1}$ be such that
$\pi^{-1}(z_i)  =  \{  \tilde z_i  \,, \, \sigma(\tilde z_i)\}$ and take $\{  \tilde z_1 ,  \sigma(\tilde z_1),  \dots  ,  \tilde z_k,  \sigma(\tilde z_k)\}$
as the set of marked points in $S_{g-1}$. Note that if $f(z_i) = z_j$, then $\widetilde f$ restricts to a bijection between the fibers
$\{  \tilde z_i  \,, \, \sigma(\tilde z_i)\}$
  and
$\{  \tilde z_j  \,, \, \sigma(\tilde z_j)\}$. 
Thus, there is a natural way to choose a lift of $f \in \diff(N_g; k)$ in a continuous manner by taking $\widetilde f \in \diff(S_{g-1}; 2k)$ to be orientation preserving. This choice defines a group homomorphism  $ \tilde{\phi}  : \diff(N_g ; k)  \to  \diff^+(S_{g-1} ;  2k)$ which induces a 
homomorphism  $\phi: \mcg(N_g ; k) \to  \mcg(S_{g-1} ; 2k)$ between the corresponding mapping class groups and makes the following diagram commutative
$$ 
\xymatrix{
\diff(N_g ;  k)  \ar[r]^{\tilde{\phi}}  \ar[d]  &  \diff^+(S_{g-1} ;  2k)   \ar[d] \\
\mcg(N_g ; k)   \ar[r]^\phi  &  \mcg(S_{g-1} ; 2k) .
}
$$ 

Here $\mcg(\Sigma; k)$ denotes the group of isotopy classes of elements of $\diff(\Sigma,k)$ (of $\diff^+(\Sigma,k)$ if the surface $\Sigma$ is orientable); 
the group $\pmcg(\Sigma; k)$ is a finite index normal subgroup of $\mcg(\Sigma; k)$.

The following result was proven in \cite[Key Lemma 2.1]{HT09} and \cite[Theorem 1.1]{GGM18}.

\begin{prop}\label{Prop:Injective}
Let $N_g$ be a non-orientable surface and let $S_{g-1}$ be its orientable double cover. The homomorphism
$ \phi  :  \mcg(N_g ; k)   \to  \mcg(S_{g-1} ; 2k) $ is injective for $g \geqslant 3$ if $k=0$ and for all $g \geqslant 1$ if $k\geqslant 1$. 
\end{prop}

\medskip

We now prove that, if $g\geqslant 2$ and $k\geqslant 1$, the group $\modn$ has $p$-periodic cohomology whenever $\modn$ has $p$-torsion.

\begin{proof}[Proof of the Theorem \ref{Thm:Main 1 pperiodicity}] 
Suppose that there exist some $g\geqslant 2$, $k\geqslant 1$ and some odd prime $p$ such that $\modn$ has $p$-torsion and is not $p$-periodic.
Then there exists $\Zp \times \Zp   \leqslant \modn$ (see \cite[Theorem X.6.7]{Brown82Coh}).  From Proposition \ref{Prop:Injective}, it follows that
$\Zp \times \Zp \cong \phi(\Zp \times \Zp )  \leqslant  \mcg(S_{g-1} ; 2k)$. 

By the Nielsen realization theorem for non-orientable surfaces (see \cite[Theorem 5.2]{Col_Xi22}), we can find $f, f'  \in \diffnk$ representing 
the generators of the subgroup $\Zp \times \Zp$ of $\modn$, such that
$$  f^p =1;  \qquad  f'^p =1;  \qquad   f f'  =  f'  f  .$$
Consider the diffeomorphisms  $\tilde{\phi}(f), \tilde{\phi}(f')\in\diff^+(S_{g-1}; 2k)$.  Since $f$ and $f'$ fix the marked points $z_i$ individually, then $\tilde{\phi}(f)$ and  $\tilde{\phi}(f')$ restrict to permutations of the fiber
$\{  \tilde z_i  \,, \, \sigma(\tilde z_i)\}$ for $i=1, \dots, k$. 
Notice that $\tilde{\phi}(f)$ and
$\tilde{\phi}(f')$ have odd order $p$, therefore they must preserve each of the marked
points of $S_{g-1}$ individually, i.e. $\tilde{\phi}(f), \tilde{\phi}(f')\in\pdiff^+(S_{g-1}; 2k)$. From the commutativity of the above diagram, it follows that
$\Zp \times \Zp \cong\phi(\Zp \times \Zp )$ is actually a subgroup of the pure mapping class group $\Gamma_{g-1}^{2k}$. This contradicts \cite[Theorem 1.7]{lu01per} which states that $\Gamma_{g-1}^{2k}$ has $p$-periodic cohomology if $g-1\geqslant 1$ and $k\geqslant 1$.
\end{proof}

\medskip

\begin{rmk} The strategy above can be used to prove $p$-periodicity in $\mathcal{N}_g$ from the cases where $\Gamma_{g-1}$ is known to be $p$-periodic  (see for instance \cite[Theorems 1, 2 and 3]{Xia92}).
However, there are cases where the $p$-periodicity of $\mathcal{N}_g$ cannot be derived from the orientable case $\Gamma_{g-1}$, as shown in \cite[Theorem 1.1 and Remark 4.4]{HT09}.
\end{rmk}

\begin{rmk}[\bf $\modno$ is not $2$-periodic]\label{Rmk:2Periodicity}
Take $S_{g-1}$ embedded in $\mathbb{R}^3$ as before and let $(x_0,0,0)\in S_{g-1}$ be a point where the surface intersects with the $x$-axis. Since the embedding is symmetric with respect to the reflection by the $yz-$plane, the point $(-x_0,0,0)$  is also in $S_{g-1}$.
Consider the rotations $R_1,R_2:S_{g-1}\to S_{g-1}$ given by,
$$R_1(x,y,z)=(-x,-y,z)\text{ \ \ \ \ and \ \ \ \ }R_2(x,y,z)=(x,-y,-z),$$
as defined in \cite[Proof of Lemma 4.1]{HT09}. These homeomorphisms are involutions and commute with the covering transformation $\sigma:S_{g-1}\to S_{g-1}$ of the orientable double cover $\pi : S_{g-1} \to N_g $. Hence, they induce $ f_1, f_2: N_g \to N_g$  and notice that $f_1, f_2 \in \diff(N_g;*)$, where the marked point is $*=[(x_0,0,0)]$. 
For genus $g\geqslant 3$, using the arguments of \cite[Proof of Lemma 4.1]{HT09}, we can see that $f_1 f_2 = f_2 f_1$, $f_1^2 = f_2^2 = id_{N_g}$, $f_1$ and $f_2$ are not isotopic to each other relative to $*$ in $N_g$ and their classes $[f_1],[f_2]\in \modno$ are non-trivial. Thus, $\langle [f_1] , [f_2] \rangle \cong \Z/2\times \Z / 2 \leqslant \modno$ and therefore $\modno$ is not $2$-periodic. 
\end{rmk}

\bigskip


\section{An upper bound for the p-period}
\label{Sec:UpperBound}

Let $p$ be an odd prime. In this section we show that, for $g\geqslant 3$ and $k\geqslant 1$, the $p$-period of the group $\modn$ is bounded above by $4$, by adapting the methods of \cite{lu01per} 
to the non-orientable case. Later on, in Section \ref{Sec:Lower Bound}, we will show that the $p$-period of $\modn$ is greater than or equal to 4 proving that $\per(\modn)=4$. 
By contrast, in the orientable case it was shown in \cite[Theorem 1.7]{lu01per} that 
the $p$-period of $\Gamma_g^k$ is equal to $2$ if $k\geqslant 1$ and $\Gamma_g^k$ contains $p$-torsion. 

We will use the Yagita invariant $Y(\modn, p)$, which can be regarded as a generalization
of the $p$-period if $\modn$ has $p$-torsion. Also, since we have already proven that $\modn$ is $p$-periodic for $g\geqslant 2$, $k\geqslant 1$, the
Yagita invariant $Y(\modn, p)$ coincides with the $p$-period of $\modn$ by \cite[Proposition 4.1.1]{Xia90The}; see also \cite{GMX94Yag} for calculations of the Yagita invariant $Y(\Gamma_g,p)$.

Recall the definition of the Yagita invariant as in \cite[Section 7]{Mis94}. Let $\Gamma$ be a group of finite virtual
cohomological dimension and $\pi \leqslant \Gamma$ any subgroup of prime order $p$. Because $\pi$ injects into any finite quotient of
the form $\Gamma/\Delta$, where $\Delta$ is a torsion-free normal subgroup of finite index in $\Gamma$, the image of the restriction map
in cohomology $H^i(\Gamma ; \Z)  \to  H^i(\pi ; \Z)$ is non-zero to some degree $i>0$.  Reduction mod-$p$ maps
$H^*(\pi ; \Z)$ onto $\mathbb F_p[u]   \subset H^*(\pi ; \mathbb F_p)$
with $u$ a generator of $H^2(\pi ; \mathbb F_p)$. Thus, there exists a maximum value $m = m(\pi, \Gamma)$ such that
$$  \Ima \left( H^*(\Gamma ; \Z)  \to  H^*(\pi ; \Z)  \right)  \subset \mathbb F_p[u^m]   \subset H^*(\pi ; \mathbb F_p) .$$
Moreover, $m(\pi, \Gamma)$ is bounded above by $m(\pi, \Gamma / \Delta)$, where $\Delta$ denotes as before a torsion-free normal subgroup of finite index. Since $\Gamma / \Delta$ is finite, it follows from the comments of \cite[Section 1]{Yag85} that $m(\pi, \Gamma)$ is bounded by a bound depending only on $\Gamma$. 
The {\em Yagita invariant} of $\Gamma$
with respect to the prime $p$ is then defined to be the least common multiple of the values $2 m(\pi, \Gamma)$, where $\pi$ ranges over
all subgroups of order $p$ of $\Gamma$. It is denoted by $Y(\Gamma, p)$.\\

We will proceed by induction on $k$, the number of marked points.
First, we prove that for every $g \geqslant 3$ the $p$-period of $\modno$ is bounded above by $4$ 
if $\modno$ has $p$-torsion. For simplicity of notation, we write $\diff(N_g; *)$ instead of the group $\diff(N_g;1)$, where $*$ will be thought of as the marked point of $N_g.$

The main idea of the proof is based on \cite[Theorem 1.4 and Theorem 1.7]{lu01per} and is outlined below. Given a subgroup $\pi \leqslant  \modno$ of order $p$, one can use Nielsen's realization theorem to obtain a lift $\widetilde \pi$  in $\diff(N_g; *)$. On the other hand, the action of $\diff(N_g; *)$ on $N_g$ induces a representation $\rho  :  \diff(N_g; *)  \to  GL_2(\R)$ given by sending a diffeomorphism $f : N_g  \to N_g$ to its differential in $*$,   $df_*  :  T_* N_g \to T_* N_g$, and which restricts to a faithful representation 
$\widetilde \rho :  \widetilde \pi  \to  GL_2^+(\R)$. The induced map at the level of classifying spaces
$B \widetilde \rho : B \widetilde \pi  \to B GL_2^+(\R)$ satisfies that there exists a class $c_1 \in H^2 (BGL_2^+(\R) ; \Z)$, essentially the first Chern class,
such that $(B \widetilde \rho)^*(c_1) \neq 0$. By diagram chasing we then obtain a class in $H^4(\modno ; \Z)$, related to the first Potryagin class in $H^4(B GL_2(\R) ; \Z)$, which restricts to the non-zero element in $H^4(B \pi ; \Z)$ corresponding to $c_1^2$. Therefore, we obtain
$m(\pi, \modno) \leqslant 2$ and the result for the $p$-period of $\modno$ follows.
Finally, we use  Birman's short exact sequence 
to obtain the result for any $k\geqslant 1$ by an induction argument on $k$. \\

We start by stating a few technical results. The following was proven in \cite[Proof of Theorem 1.4]{lu01per}.

\medskip

\begin{lem}\label{Lemma:Faithful_non_zero} 
	Let $\rho : \pi \to GL_2^+(\R) \simeq SO(2)$ be a faithful representation of a  non-trivial cyclic group $G$. Then there exists $c_1 \in H^2( BGL_2^+(\R);\Z ) \cong H^2(BSO(2) ;\Z)$ such that    $$(B \rho)^* (c_1)\neq 0  \ \  \text{ in }  \ \ H^2(B\pi ;\Z ) .$$
\end{lem}

\bigskip

\begin{lem}\label{Lemma:inclusion_H4}
	There exists a non-zero element $p_4 \in H^4(BO(2);\Z )$ such that $$(B\iota)^* (p_4)={c_1}^2  \ \  \text{ in }  \ \ H^4(BSO(2);\Z ) , $$ where $\iota$ is the canonical inclusion $SO(2) \hookrightarrow O(2) $.
\end{lem}
\begin{proof}
    Consider the following commutative triangle
 $$ \xymatrix@C=2cm @M=2mm {
	SO(2) \ar[d]_-{\iota} \ar[rd]^-{i}  &   \\
	O(2) \ar[r]^-{\rho} & SO(3)   
	} $$
 where $\rho:O(2) \to SO(3)$ is the homomorphism given by 
$$\rho (A) = \left( \begin{matrix}
A & 0 \\ 
0 & det(A)
\end{matrix} \right)  .$$ 
and $i:SO(2)\to SO(3)$ is the natural inclusion. 
Consider the induced map $Bi: BSO(2) \to BSO(3)$ of classifying spaces, this map is a fibration with fiber $S^2$. Recall from \cite[Theorem 1.5]{Brownj82}, \cite[Theorem 1]{Fesh83} that the integral cohomology of $BSO(3)$ is given as a graded algebra by
$$H^*\left( BSO\left( 3\right) ;\mathbb{Z} \right) \, \cong \, \Z [ v_3, p' _4 ]  \, \big /  \, \langle 2v_{3} \rangle ,$$
where the subscripts indicate the degree of each generator. On the other hand, the cohomology of $BSO(2)$ is given by 
$$H^*(BSO(2) ; \Z) = H^*(BU(1); \Z) = \Z[c_1],$$
where $c_1$ denotes the first Chern class. By a straightforward analysis of the Serre spectral sequence, we can see that the induced homomorphism in $H^4(-;\Z)$
$$(Bi)^* : H^4 (BSO(3) ; \Z ) \to H^4 (BSO(2) ;\Z)$$ 
is an isomorphism. Thus, there exists $p_4 '' \in H^4(BSO(3);\Z)$ such that 
$(Bi)^*(p_4'')={c_1}^2$.
Therefore, $p_4:=(B \rho) ^* (p_4 '' )  \in H^4( B O(2) ; \Z )$ is the desired cohomology class.
\end{proof}

We now proceed to prove our main result of the section.

\medskip

\begin{thm}[{\bf Upper bound for the $p$-period of  $\mathcal{N}_{g}^{1}$}]\label{Thm:UpperBound OnePoint}
	Let $g\geqslant 3$, and $p$ be an odd prime. If $\modno$ contains $p$-torsion, then the $p$-period of $\modno$ is lower or equal than $4$, i.e., $\per({\modno})\leqslant 4 $.
\end{thm}
\begin{proof}
We will show that the Yagita invariant $Y(\modno,p)\leqslant 4$. Let $\pi\leqslant \modno$ be a subgroup of order $p$. By the Nielsen realization theorem for non-orientable surfaces (see \cite[Theorem 5.2]{Col_Xi22}), there exists a subgroup $\widetilde{\pi} \leqslant \diff(N_g;*)$ such that $\widetilde{\pi}\cong \pi$.
Thus there is a commutative diagram
$$ \xymatrix@C=1.5cm @M=2mm {
\widetilde \pi \ar@{^{(}->}[r]^-{\widetilde i} \ar[d]_-{\cong} & \diff(N_g; *) \ar[d] \\
\pi \ar@{^{(}->}[r]^-i & \modno.  
} $$

Hence, $\widetilde \pi$ acts on the surface $N_g$ as a group of diffeomorphisms that fix $*$. We obtain the following representation 
which arises from letting a diffeomorphism of $(N_g, *)$ act on $T_* N_g$, the unoriented tangent space of $N_g$ at $*$
$$ \rho: \diff (N_g, * ) \longrightarrow GL_2(\R), \qquad  f \longmapsto df_*.$$ 

Consider the representation of $\widetilde \pi$ given by the composition
$$\rho \circ \widetilde i: \widetilde{\pi} \overset{\widetilde i}{\longrightarrow} \diff(N_g; *) \overset{\rho}{\longrightarrow} GL_2 (\R).$$ 
Since $\widetilde \pi$ is a cyclic group of odd order $p$, we have $\Ima ( \ \rho \circ \widetilde i \ ) \subset GL_2^+ (\R)$. 
Denote by $\widetilde{\rho} : \widetilde{\pi} \to GL_2^+(\R)$ the resulting representation by restricting the image. 
Being $N_g$ non-orientable, it can be given a dianalytic structure of a Klein surface
in which $\widetilde \pi$ acts as a group of rotations on a neighborhood of $*$ with respect to this
structure. Thus, $\widetilde \rho$ is a faithful representation. By Lemma \ref{Lemma:Faithful_non_zero} there exists a class $c_1 \in H^2( BGL_2^+ ; \Z)$ such that
\begin{equation}\label{Eqn:Thm-Per1}
	v:= (B\widetilde{\rho})^*(c_1) \neq 0   \quad \text{in}  \quad H^2(\widetilde{\pi} ; \Z).
\end{equation}

This information can be summarized in the following commutative diagram.

$$
\xymatrix{
GL_2^+(\R) \ar[r]^-{\iota}  & GL_2(\R)   & df_* :  T_* N_g \to T_* N_g   \quad\\
\widetilde\pi  \ar[d]_-{\cong}  
\ar@{^(->}[r]<-0.5ex>^-{\widetilde i}
\ar[u]^{\widetilde \rho}^{\text{
\begin{minipage}{2cm}  
Faithful \\ \hspace*{0.15cm} rep.
\end{minipage}
}} &  \diff(N_g; *)  \ar[d]  \ar[u]_\rho \ar[d]&  f \ar@{|->}[u]<9ex> \qquad  \qquad  \qquad   \qquad \;\;\\
\pi \ar@{^(->}[r]<-0.5ex>^{i} &  \modno  &
}
$$


Notice that the natural inclusions induce a homotopy commutative diagram at the level of classifying spaces 
$$
\xymatrix{
B SO(2)    \ar[r]^{B \iota}  \ar[d]_\simeq  &  B O(2) \ar[d]^\simeq \\
BGL_2^+(\R)  \ar[r]^{B \iota} &   B GL_2(\R)  
}
$$
where the vertical maps are homotopy equivalences. 
Thus, by the Lemma \ref{Lemma:inclusion_H4} there exists a class
$p_4 \in H^4 ( BGL_2 ( \R);\Z)$ such that
\begin{equation}\label{Eqn:Thm-Per2}
		(B \iota) ^*(p_4) = {c_1}^2 .
\end{equation}	
Passing to cohomology of classifying spaces

\bigskip

$$
\xymatrix @C=2mm{
 c_1 \ar@{|->}[d] & \overset{\text{{\small ${c_1}^2$}}}{H^* BGL_2^+(\R)}  \ar[d] & &  
H^* BGL_2(\R) \ar[ll]_-{(B\iota)^*}  \ar[d]  &  {p_4}
\ar@/_2pc/[lll]  
\ar@/^1pc/[dd]\\
 v \neq 0 & H^*(B \widetilde \pi) &  &  H^*B \diff(N_g;*) \ar[ll]_-{(B\widetilde i \,)^*}&  \\
& \underset{\text{ \small  {$0 \neq\bullet$}}}{H^*(B \pi)}  \ar[u]^{\cong} & & H^*(B \modno) \ar[ll]_-{(Bi)^*}  \ar[u]_-{\cong} &  
{\bullet}
\ar@/^2pc/[lll]\\
}
$$


\vspace{8mm}

Since the identity component of the group $\diff (N_g;*)$ for $g\geqslant 3$ is contractible by \cite[Prop. 2 and Thm. 2]{Gram73}, it follows that the induced map in classifying spaces $B\diff (N_g;*) \to B\modno $ is a homotopy equivalence.
Thus, the vertical right arrow in the bottom square is an isomorphism. This argument exhibits a class in $H^4(B\modno ; \Z)=H^4(\modno ; \Z)$ (namely the image of $p_4$) that maps 
to a non-zero class under the restriction
$$ H^*(\Gamma ; \Z)     \to  H^*(\pi ; \Z)   \xrightarrow{ \,{\rm mod} \, p \,}     \mathbb F_p[u].  $$
Therefore, $m(\pi ;  \modno) \leqslant 2$ and thus $\per(\modno) \leqslant 4$. 
\end{proof}

Finally, we use the Birman exact sequence (see \cite[Proposition 1 and Lemma 1]{Gram73} and \cite[Theorem 2.1]{Kork02})
to generalize the previous result from $\modno$ to $\modn$  for any $k \geqslant 1$ and $g \geqslant 3$.

\medskip

\begin{thm}[{\bf Upper bound for the $p$-period of  $\mathcal{N}_{g}^{k}$}]\label{Thm:UpperBound} Let $k \geqslant 1$, $g \geqslant 3$  and  $p$ an odd prime. If the group $\modn$ contains $p$-torsion, then the $p$-period of $\modn$ is $\per(\modn)\leqslant 4$.
\end{thm}
\begin{proof}
The proof is by induction on the number of marked points. The case $k=1$ is precisely Theorem \ref{Thm:UpperBound OnePoint}. Assume that the result holds for the case of $k\geqslant 1$ and suppose that
$\mathcal{N}_g^{k+1}$ has $p$-torsion. Since $g \geqslant 3$ we can consider the Birman exact sequence
$$1 \to \pi_1 (N_g^{k}) \to \mathcal{N}_g^{k+1} \to \modn \to 1,$$
where $N_g^{k}$ denotes the surface obtained from $N_g$ by removing the $k$ marked points. Since $\pi_1 (N_g^{k})$ is a free group, it follows that $\modn$ must contain $p$-torsion. Thus, from Theorem \ref{Thm:Main 1 pperiodicity} we see that $\modn$ has $p$-periodic cohomology and by the induction hypothesis we have $\per(\modn)\leqslant 4$. Since $\vcd(\modn)$ and $\cd(\pi_1 (N_g^{k}))$ are finite, it follows from \cite[Lemma 1.1]{lu01per} that $\mathcal{N}_g^{k+1}$ has $p$-periodic cohomology and $\per(\mathcal{N}_g^{k+1})|\per(\modn)$. Consequently, $\per(\mathcal{N}_g^{k+1}) \leqslant 4$.
\end{proof}

\medskip


\section{A lower bound for the p-period}
\label{Sec:Lower Bound}

In this part, we find a lower bound for the $p$-period of $\modn$. Our approach uses the following result, which can be deduced by
\cite[Lemma 3.1]{GMX92} and the Brown decomposition Theorem \cite[Corollary X.7.4]{Brown82Coh}.

\medskip

\begin{thm} \label{Thm:GMX pperiod}
Suppose $\Gamma$ has finite $\vcd$ and $p$-periodic cohomology. Moreover, assume that $\Gamma$ contains only finitely many conjugacy classes of subgroups of order $p$. Then the $p$-period of $\Gamma$ is given by 
$$\per(\Gamma)=2\cdot {\lcm} \{ [N_\Gamma(\Zp) : C_\Gamma(\Zp) ] \mid \Zp\in S \} \cdot p^d $$
for some integer $d\geqslant 0$, where $S$ is a set of representatives of the conjugacy classes of subgroups of $\Gamma$ of order $p$, $N_\Gamma(\Zp) $ and $C_\Gamma (\Zp)$ are the normalizer and centralizer of $\Zp$ in $\Gamma$ respectively. 
\end{thm}

\bigskip

To apply the above result, first observe that if $\Zp = \langle \alpha \rangle$, then for each $ \beta\in N_{\Gamma}(\Zp) $  
there exists $0< m_{\beta} < p $ such that $ \beta \alpha \beta^{-1} = \alpha^{m_{\beta}} $. This allows us to define a bijection 
\begin{align*}
	N_{\Gamma} ( \Zp )/C_{\Gamma} ( \Zp ) & \to  \{ m \in \{1,\ldots,p-1\} \mid \alpha^m \text{ is conjugate to } \alpha \}  \\
	[\beta] & \mapsto m_\beta. 
\end{align*}

Therefore, if we determine all the powers to which one of the possible generators of the subgroup $\Zp$ is conjugated, we can obtain information about the $p$-period of the group $\Gamma$. 
According to the preceding results, in the case of $\modn$ where $k\geqslant 1$ and the group has $p$-torsion, it suffices to find at least one nontrivial power $\alpha^m$ such that $\alpha$ and $\alpha^m$ are conjugated in $\modn$ to conclude that $\per(\modn)=4$.  In the case $k=0$, this nontrivial power gives us a lower bound for the $p$-period $\per(\mathcal{N}_g)\geqslant 4$,  for the genera $g\geqslant 3$ when the group $\mathcal{N}_g$ has $p$-periodic Farrell cohomology. 
However, this power is difficult to find directly in the mapping class group $\modn$.  Fortunately, the Nielsen realization theorem allows us to work with elements in the group $\diffnk$ rather than mapping classes in $\modn$. \\

Let us denote by $\homek$  the group of homeomorphisms of $N_g$ that fix the marked points pointwise. The proof of the following result is postponed to the end of this section, since new tools need to be introduced.

\medskip

\begin{thm}\label{Thm:f y f-1 conjugados}
    Let $g\geqslant 3$ and $k\geqslant 0$. If $f\in \diffnk$ is of order $p$, then $f$ and $f^{-1}$ are conjugated in $\homek$.
\end{thm}

\medskip

\begin{rmk}
Notice that in the case where the surface is orientable, a result such as Theorem \ref{Thm:f y f-1 conjugados} does not hold in general. For instance, if there there is at least one marked point, from Theorem \ref{Thm:GMX pperiod} and \cite[Theorem 1.4]{lu01per} it follows that such a result can never occur.  
\end{rmk}

\medskip

Having this result, we can prove one of our main results.

\medskip

\begin{thm}[{\bf Lower bound for the $p$-period of  $\modn$}]\label{Thm:LowerBound}
	Let $k\geqslant 0$, $g\geqslant 3$ and $p$ be an odd prime. If the group $\modn$ has $p$-periodic Farrell cohomology, then the $p$-period is bounded below by $4$, that is, $\per (\modn ) \geqslant 4$.
\end{thm}

\begin{proof}
    Let $\Zp= \langle \alpha \rangle $ be a subgroup of prime order $p$ in $\modn$. By the above discussion, we have a bijection 
    \begin{align*}
	   N_{\modn} ( \Zp )/C_{\modn} ( \Zp ) & \to  \{ m \in \{1,\ldots,p-1\} \mid \alpha^m \text{ is conjugate to } \alpha \}.  
    \end{align*}
    
     By the Nielsen realization theorem (see \cite{Col_Xi22}*{Theorem 5.2}) there exists a diffeomorphism $f\in \diffnk$ such that $f\in \alpha$ and $f^p=1$. By Theorem \ref{Thm:f y f-1 conjugados} we have that $f$ and $f^{-1}$ are conjugated in $\homek$. Thus, there exists $s\in \homek$ such that $s f s^{-1} = f^{-1}$. 
     Let $\beta$ denote the image of $s$ under the canonical projection $\homek \to \modn$, then $ \beta \alpha \beta^{-1}= \alpha^{-1}.$ 
     Thus, $\alpha$ is conjugated to $\alpha^{-1}$ in $\modn$, which implies that $[N_{\modn}(\Zp) : C_{\modn}(\Zp) ] \geqslant 2 $. 
     
     By Theorem \ref{Thm:GMX pperiod}, we have the following expression of the $p$-period
    $$\per(\modn)=2\cdot \lcm \{ [N_{\modn}(\Zp) : C_{\modn} (\Zp) ] \mid \Zp\in S \} \cdot p^d, $$
    where $S$ is a set of representatives of the conjugation classes of subgroups of order $p$ and $d\geqslant 0$. Thus, the formula entails
    $$\per(\modn) \geqslant 4 \cdot p^d \geqslant 4 ,$$
    which proves the result.
\end{proof}

\bigskip


\n
\textbf{Proof of the Theorem \ref{Thm:f y f-1 conjugados}.}

\medskip

In this section, we complete the proof of our main theorem by proving that $f$ and $f^{-1}$ are conjugated in $\homek$, whenever $f \in \diffnk $ is of order $p$.
For this purpose, we will use non-Euclidean crystallographic groups (\textit{NEC groups}, for short), which are similar to Fuchsian groups, but orientation-reversing isometries of the hyperbolic plane $\hiper$ are allowed. For a general discussion of this topic, we refer the reader to \cite[Section 0.2]{Bujalance90} and \cite[Section 1]{Bujalance10}.
The surface $N_g$ will be uniformized by a NEC group $K$ such that $f:\hiper/K \to \hiper/K$ is a dianalytic map.
To find a homeomorphism that conjugates the diffeomorphisms, $f$ and $f^{-1}$, we construct an automorphism of $K$ that connects $f$ and its inverse. 
The advantage of using NEC groups lies in the fact that every automorphism of a NEC group can be realized geometrically (see \cite[Theorem 3]{Macbeath67}). In this way, we can find a homeomorphism $ \tau:\hiper \to \hiper$ such that when we use the universal cover $q_K: \hiper \to \hiper / K$, this induces a homeomorphism $s: \hiper / K \to \hiper / K $ with the desired property $s\circ f\circ s^{-1}=f^{-1}$. 

\bigskip

\n
\textbf{Surface-kernel epimorphism.} Let $f\in \diffnk$ be an element of order $p$. By \cite[Proposition 1]{Tucker83}, the lifting $\widetilde{f}:=\tilde{\phi}(f)\in \diff (S_{g-1};2k)$ has a finite number of fixed points, thus $f$ also has a finite number of fixed points. Denote the fixed points of $f$ by $z_1,\ldots, z_t \in N_g$, with the convention that the first $k$ points are the marked points. We can endow the surface $N_g$ with a dianalytic structure $\X$ such that $\langle f \rangle$ is a group of automorphisms of the Klein surface $(N_g;\X)$ or, in other words, the mapping 
$$f:(N_g,\X)\to (N_g,\X)$$
is dianalytic. By the uniformization theorem of Klein surfaces, there exists a non-Euclidean crystallographic group $K$ isomorphic to the fundamental group of $N_g$ such that the quotient surface $\hiper / K$ is isomorphic to $N_g$ as Klein surfaces. 
Let $\gamma:\hiper \to \hiper$ be the lifting of the diffeomorphism $f:N_g\to N_g$ to the universal cover $q_K: \hiper \to \hiper / K$. Since $f$ is dianalytic, then $\gamma:\hiper \to \hiper$ is an isometry and this allows us to define
the NEC group
$$\Gamma:=\langle K, \gamma \rangle.$$
Therefore the quotient space $\hiper / \Gamma$ is homeomorphic to $N_g / \langle f \rangle$ which in turn is homeomorphic to $N_h$, where $h$ satisfies the Riemann-Hurwitz equation,
$$g-2=p(h-2)+t(p-1),$$
and $h\geqslant 1$ and $t\geqslant k$. On the other hand, it is not hard to see that $K \lhd \Gamma$, thus, we have a ramified covering $q:\hiper/ K \to \hiper / \Gamma$ and for every element $u\in \Gamma$ there exists an induced homeomorphism $\widehat{u}: \hiper /K \to \hiper / K $ defined for each $\zeta \in \hiper $ as 
$$ \zeta \cdot K \mapsto u(\zeta) \cdot K, $$
which makes the following diagram commutative
$$\cuadro{\hiper}{u}{\hiper}{}{}{N_g\cong \hiper/K}{\widehat{u}}{N_g\cong\hiper/K.}$$

We can see that $\widehat{u}:\hiper/K \to \hiper/K $ is a covering transformation for the ramified cover $q:\hiper / K \to \hiper / \Gamma $. We now define the epimorphism $\theta: \Gamma \to \langle f \rangle\cong\mathbb{Z}/p$ given by
$$\theta(u)=\widehat{u} \text{ \ \ \ for all \ \ \ } u\in \Gamma. $$
and in this way, a short exact sequence is obtained
$$1 \to K \to \Gamma \xrightarrow{\theta} \langle f \rangle \to 1.$$

In the literature,  the epimorphism $\theta$ is called \textit{smooth} or \textit{surface-kernel epimorphism}; see for example \cite[Section 1.4]{Bujalance10}. From now on, we will consider the surface $N_g$ as the quotient $\hiper / K$ and its elements are represented by $\zeta \cdot K$.

\bigskip

\n
\textbf{Canonical presentation of the NEC group $\Gamma$.}
Since the canonical projection $q_\Gamma:\hiper \to \hiper / \Gamma$ has $t $ branched points corresponding to the fixed points of $f$, whose ramification index is equal to $p$ and $\hiper / \Gamma$ is homeomorphic to $N_h$, then $\Gamma$ has an algebraic presentation given by (see \cite[Proposition 1.1.4]{Bujalance10}) 
\begin{equation}\label{Eqn:Canonical Pres de Gamma}
    \langle x_1 , \ldots , x_t, d_1, \ldots , d_h \mid x_1 \cdot \ldots \cdot x_t\cdot d_1^2 \cdot \ldots \cdot d_h^2 = x_1^p = \ldots = x_t^p = 1 \rangle,    
\end{equation}
where $x_1,\ldots , x_t$ are elliptic elements of $\Aut(\hiper)$ and $d_1,\ldots , d_h$ are glide reflections of $\Aut(\hiper)$. 
The above presentation will be called \textit{a canonical presentation} of $\Gamma$ and the generators will also be called \textit{canonical generators.} 
With this presentation, notice that the elliptic generators contain, in some sense, the information about the fixed points of $f$.

\medskip

\begin{rmk}\label{Rmk:Elliptic generators and fixed points}
    Consider the canonical presentation \eqref{Eqn:Canonical Pres de Gamma} of the group $\Gamma$. Let $\zeta_1, \ldots, \zeta_t \in \hiper$ be the fixed points of the elliptic canonical generators $x_1,\ldots, x_t$, respectively. By definition of $\theta: \Gamma \to \langle f \rangle $, we have that  
    $$ \theta(x_i) (\zeta_i \cdot K) := x_i(\zeta_i) \cdot K= \zeta_i \cdot K$$  and  $\theta(x_i) = f^{m_i} $ for some $1\leqslant m_i\leqslant p-1$.  Since $(m_i,p)=1$, we have that $ \zeta_i \cdot K $ is also a fixed point of $f$, for each $i=1,\ldots , t$.
\end{rmk}

\medskip

\begin{rmk}\label{Lem:Fixed point differs by Gamma}
    Suppose that $\zeta\cdot K \in \hiper/ K$ is a fixed point of $f$. Then for each $u\in \Gamma$ we have that
    $$u(\zeta)\cdot K = \zeta \cdot K .$$
\end{rmk}

\bigskip

\n
\textbf{Auxiliary epimorphisms $\theta_1$ and $\theta_2$.}
Given an element $u\in \Gamma$ it is clear that 
$$\theta(u)= f^{m_u}=(f^{-1})^{-m_u}\text{, for some $1 \leqslant m_u \leqslant p$}.$$
We will consider $m_u$ as an element of $\Zp$, to avoid problems with the range from which we can select $m_u$. We define the epimorphisms $\theta_1 : \Gamma \to \Zp$ and $\theta_2 :\Gamma \to \Zp$ given by $\theta_1(u)=m_u$ and $\theta_2(u)=-m_u$. Observe that, by definition 
\begin{align*}
    \theta(u)=&f^{\theta_1(u)} & & \text{ and }  &  \theta(u)=& (f^{-1})^{\theta_2(u)}
\end{align*}
where, abusing the notation $\theta_1(u)$ and $\theta_2(u)$ are thought of as integers and not as classes of $\Zp$. With this observation, we can see that the epimorphisms $\theta_1$ and $\theta_2$ give us a distinction on how to take a preferred generator from the group of covering transformations of $q:\hiper / K \to \hiper / \Gamma$. Namely, for $\theta_1$ we take $f$ as the preferred generator while for $\theta_2$ we take $f^{-1}$. \\

The following result connects the two epimorphisms $\theta_1$ and $\theta_2$ by an isomorphism $\psi:\Gamma \to \Gamma$. This result, combined with the previous discussion that $\theta_1$ and $\theta_2$ contain information on the choice of a preferred generator of $\Zp$ (which are $f$ and $f^{-1}$), gives us the guideline to prove that $f$ and $f^{-1}$ are conjugate, as will be seen in later results.

\begin{lem}\label{Lem:Isomorphism of Gamma}

 For the epimorphisms $\theta_1: \Gamma \to \Zp$ and $ \theta_2: \Gamma \to \Zp$ there exists an isomorphism $\psi: \Gamma \to \Gamma$ such that the following diagram is commutative:
\begin{equation}\label{Eqn:Diagram connection}
    \xymatrix @R=3mm @C=10mm {
\Gamma \ar[rd]^-{\theta_1} \ar[dd]_-{\psi} & \\
    &   \Zp \\
\Gamma \ar[ru]_-{\theta_2} & 
}
\end{equation}
Moreover, if $\zeta_i\in \hiper$ is the fixed point of the elliptic generator $x_i$, then there exists $u_i\in \Gamma$ such that the elliptic generator $\psi(x_i)$ has  $u_i (\zeta_i)\in \hiper $ as  a fixed point.
\end{lem}

\begin{proof}
Consider the canonical presentation \eqref{Eqn:Canonical Pres de Gamma} of $\Gamma$  and define the following elements of $\Gamma$:
\begin{align*}
    \eta = & x_1\cdot \ldots \cdot x_t \cdot d_1, & &  \\
    \chi_i = & x_{i+1}\cdot x_{i+2}\cdot \ldots \cdot x_t & \text{for } & i=1,\ldots, t-1, &  \chi_t = & 1,  \\
    \delta_j = & d_{j+1}^2\cdot d_{j+2}^2\cdot \ldots \cdot d_h^2 & \text{for } & j=1,\ldots, h-1, & \delta_h = & 1.     
\end{align*}
We define the function $\psi: \Gamma \to \Gamma$ at the level of generators as follows:
\begin{align*}
    x_i & \mapsto \eta \cdot \chi_i \cdot x_i^{-1} \cdot \chi_i^{-1} \cdot \eta^{-1} & \text{for } & i=1,\ldots t \\  d_1 & \mapsto \eta^2 \cdot \delta_1 \cdot \eta^{-1} & & \\
    d_j & \mapsto \eta \cdot \delta_j \cdot d_j^{-1} \cdot \delta_j^{-1} \cdot \eta^{-1} & \text{for } & j=2,\ldots h. 
\end{align*}

Since the relation of the group $\Gamma$ is preserved by $\psi$, it follows that $\psi $ defines a group homomorphism from $\Gamma$ to $\Gamma$. Moreover, it can be checked that  $\psi$ is  an isomorphism. 

On the other hand, notice that $\psi$ satisfies the following properties:
\begin{align*}
    \theta_2(\psi (x_i ))= & \theta_2( \eta \cdot \chi_i \cdot x_i^{-1} \cdot \chi_i^{-1} \cdot \eta^{-1} ) = -\theta_2(x_i) & \text{for } & i=1,\ldots t, \\
    \theta_2(\psi(d_1))= & \theta_2( \eta^2 \cdot \delta_1 \cdot \eta^{-1} ) = \theta_2(d_1) + \sum_{i=1}^{t} \theta_2(x_i) + 2\cdot \sum_{j=2}^{h}  \theta_2(d_j) & & \\
    \theta_2(\psi(d_j)) = & \theta_2( \eta \cdot \delta_j \cdot d_j^{-1} \cdot \delta_j^{-1} \cdot \eta^{-1} ) = -\theta_2( d_j ) & \text{for } & j=2,\ldots h. 
\end{align*}

Since the generators of the group $\Gamma$ satisfy the relation $ x_1 \cdot \ldots \cdot x_t\cdot d_1^2 \cdot \ldots \cdot d_h^2 = 1$, it follows that 
$$ - \theta_2(d_1) = \theta_2(d_1) + \sum_{i=1}^{t} \theta_2(x_i) + 2\cdot \sum_{j=2}^{h} \theta_2(d_j),$$
which implies that 
\begin{align*}
	\theta_2(\psi(x_i))=& -\theta_2(x_i)=\theta_1(x_i)   	&  &\text{for all \ \ \ } i=1,\ldots, t, \\
	\theta_2(\psi(d_j))=& - \theta_2(d_j)=\theta_1(d_j) 	&  &\text{for all \ \ \ } j=1, 2,\ldots, h,
\end{align*}
since, by the definition of $\theta_2:\Gamma \to \Zp$, we have $\theta_1(u)=-\theta_2(u)$ for all $u\in \Gamma$. 
It follows that the condition holds for all $u\in \Gamma$, that is, $\theta_2(\psi(u))=\theta_1(u)$. Therefore, $\psi:\Gamma \to \Gamma$ is the desired isomorphism.

Finally, by the definition of $\psi$, for each $i=1,\ldots , t$, if $\zeta_i$ is the fixed point of the elliptic generator $x_i$, then $ u_i(\zeta_i) \in \hiper $ is the fixed point of $\psi(x_i)$, where $u_i= \eta \cdot \chi_i $. This completes the proof.
\end{proof}


\bigskip

We now proceed to prove that if $f\in \diffnk$ is of order $p$, then $f$ and $f^{-1}$ are conjugated in $\homek$.

\medskip

\begin{proof}[Proof of the Theorem \ref{Thm:f y f-1 conjugados}]
Let $K$ be the NEC surface group such that $f: \hiper / K\to \hiper / K $ is an isometry and $\gamma:\hiper\to \hiper$ the lifting of $f$ to the universal cover $q_K:\hiper \to \hiper / K $. Consider the NEC group $\Gamma=\langle K,\gamma \rangle $, the surface kernel epimorphism $\theta: \Gamma \to \langle f \rangle$ and the two auxiliary epimorphisms $\theta_1,\theta_2:\Gamma \to \Zp$ defined above. By Lemma \ref{Lem:Isomorphism of Gamma}, we can construct an isomorphism $ \psi:\Gamma\to \Gamma $ such that the diagram \eqref{Eqn:Diagram connection} is commutative and if the fixed points of the elliptic generators are $\zeta_i\in \hiper$, then the fixed points of $\psi(x_i)$ are equal to $u_i (\zeta_i)\in \hiper$, for some $u_i \in \Gamma$. By \cite[Theorem 3]{Macbeath67}, the isomorphism $\psi:\Gamma \to \Gamma$ is realized geometrically, this means that there exists a homeomorphism $\tau:\hiper \to \hiper$ such that
\begin{equation}\label{Eqn:Geometric realization psi}
	\psi(u)=\tau u \tau^{-1} \text{ \ \ \ for all \ \ \ } u \in \Gamma.
\end{equation}

Now, by the commutativity of the diagram \eqref{Eqn:Diagram connection}, we can see that $\psi(\ker(\theta_1))=\ker(\theta_2)$. But $K=\ker (\theta_1)= \ker(\theta_2)$, which implies that $\psi|_K:K\to K$ is an automorphism of $K$. Thus, $\tau: \hiper \to \hiper$ induces the following homeomorphisms
\begin{align*}
\widehat{s}:\hiper/\Gamma & \to \hiper / \Gamma  & s:\hiper/ K & \to \hiper / K \\
 \zeta \cdot \Gamma & \mapsto  \tau(\zeta) \cdot \Gamma  &  \zeta \cdot K\ & \mapsto  \tau(\zeta) \cdot K,
\end{align*}
and these are such that the following diagram is commutative
$$\xymatrix @C=1cm @M=2mm{
	\hiper \ar[r]^-{\tau}  \ar[d] & \hiper \ar[d]  \\
	\hiper/ K \ar[r]^-{s} \ar[d] & \hiper / K \ar[d]\\
	\hiper/ \Gamma \ar[r]^-{\widehat{s}} & \hiper / \Gamma. 	
	 }$$

By definition of the epimorphism $\theta: \Gamma \to \langle f \rangle$ and using the above diagram, it follows that 
$\theta(\tau \gamma \tau^{-1})=s \circ f \circ s^{-1}.$
On the other hand, by definition of $\theta_2: \Gamma \to \Zp$, we have that $\theta(\tau \gamma \tau^{-1})=(f^{-1})^{\theta_2(\tau \gamma \tau^{-1})}$, but from the diagram \eqref{Eqn:Diagram connection} and the equation \eqref{Eqn:Geometric realization psi} it follows that
$ \theta_2 (  \tau \gamma \tau^{-1}  ) = \theta_2(\psi(\gamma))= \theta_1 (\gamma) = 1, $
therefore
$$ s \circ f \circ s^{-1}=f^{-1} .$$

It remains to prove that $s\in \homek$. According to the Lemma \ref{Lem:Isomorphism of Gamma}, for each $i=1,\ldots, \ldots ,t$ if $\zeta_i\in \hiper $ is the fixed point of the elliptic generator $x_i$, then the fixed point of $\psi(x_i)$ is $u_i ( \zeta_ i )$ for some $u_i\in \Gamma $. 
Applying $\psi$ to each of the elliptic generators and by the equation (\ref{Eqn:Geometric realization psi}) we have that
$ \psi(x_i) = \tau \ x_i \ \tau^{-1} .$
Thus, $ \tau ( \zeta_ i )  $ is a fixed point of $\psi(x_i)$. Since the elliptic transformation $\psi(x_i)$ only has one fixed point in $\hiper$, it follows that 
$\tau ( \zeta_ i ) =u_i ( \zeta_ i ).$
This implies, by Remark \ref{Lem:Fixed point differs by Gamma} and the definition of $s$ that 
    $$ s(\zeta_i \cdot K) =  \tau (  \zeta_ i  ) \cdot K = u_i(\zeta_i ) \cdot K=  \zeta_i \cdot K.$$

Moreover, as we pointed out in  Remark \ref{Rmk:Elliptic generators and fixed points}, the points $\zeta_i \cdot K$ are the fixed points of $f$. Hence, the marked points of $N_g$ remain fixed by $s$. Therefore the homeomorphism $s\in \homek$, which completes the proof.
\end{proof}

\bigskip

\section*{Acknowledgements}
This work started as part of the Ph.D. dissertation of the first author. He would like to thank the other two authors for their support and valuable comments while this work was conducted, and the financial support of CONACYT through the Ph.D. scholarship No. 494867. The first two authors acknowledge funding from CONAHCYT grant CF 2019-217392. All authors are grateful for the financial support of CONAHCYT grant CB-2017-2018-A1-S-30345.

\bigskip

\bibliography{References}

\end{document}